\documentclass[11pt,oneside,english]{amsart}
\usepackage[T1]{fontenc}
\usepackage[latin9]{inputenc}
\usepackage{babel}
\usepackage{prettyref}
\usepackage{amsthm}
\usepackage{amssymb}
\usepackage[unicode=true,pdfusetitle,
 bookmarks=true,bookmarksnumbered=false,bookmarksopen=false,
 breaklinks=false,pdfborder={0 0 1},backref=false,colorlinks=false]
 {hyperref}
\usepackage{breakurl}

\makeatletter
\numberwithin{equation}{section}
\numberwithin{figure}{section}
  \theoremstyle{definition}
  \newtheorem{defn}{\protect\definitionname}
\theoremstyle{plain}
\newtheorem{thm}{\protect\theoremname}
  \theoremstyle{plain}
  \newtheorem{lem}{\protect\lemmaname}

\usepackage{fullpage}

\makeatother

  \providecommand{\definitionname}{Definition}
  \providecommand{\lemmaname}{Lemma}
\providecommand{\theoremname}{Theorem}

\begin{document}

\title{On the norm of inverses of confluent Vandermonde matrices}

\author{Dmitry Batenkov}

\address{Department of Mathematics\\
Weizmann Institute of Science\\
Rehovot 76100\\
Israel}

\email{dima.batenkov@weizmann.ac.il}

\urladdr{http://www.wisdom.weizmann.ac.il/~dmitryb}



\thanks{This research has been supported by the Adams Fellowship Program
of the Israel Academy of Sciences and Humanities.}

\date{\today}
\begin{abstract}
In this note we present a simple upper bound for the row-wise norm
of the inverses of general confluent Vandermonde matrices.
\end{abstract}
\maketitle

\section{Introduction and main result}

Let $\{x_{1},\dots,x_{n}\}$ be pairwise distinct complex numbers
and $\{\ell_{1},\dots,\ell_{n}\}$ a vector of natural numbers such
that $\ell_{1}+\ell_{2}+\dots+\ell_{n}=N$.
\begin{defn}
The $N\times N$ \emph{confluent Vandermonde matrix} is
\[
V=\begin{bmatrix}v_{1,0} & v_{2,0} & \dots & v_{n,0}\\
v_{1,1} & v_{2,1} & \dots & v_{n,1}\\
\vdots\\
v_{1,N-1} & v_{2,N-1} & \dots & v_{n,N-1}
\end{bmatrix}
\]
where $v_{j,k}=\begin{bmatrix}x_{j}^{k}, & kx_{j}^{k-1}, & \dots & k(k-1)\times\dots\times(k-\ell_{j}+1)x_{j}^{k-\ell_{j}+1}\end{bmatrix}$.
\end{defn}
While the usual Vandermonde matrices, corresponding to the configuration
$\ell_{1}=\dots=\ell_{n}=1$, are ubiquitous, the general confluent
case is somewhat less known. The confluent matrices classically appeared
in theory of interpolation and quadrature \cite{higham1996accuracy,kalman1984gvm},
as well as in more recent studies of higher-order numerical methods
in signal processing \cite{batFullFourier,batenkov2011accuracy,batyomAlgFourier}.

It is often desirable to estimate the row-wise norm of $V^{-1}$ (which
can be used to further evaluate the condition number), see e.g. \cite{bazan2000crv,beckermann2000condition,gautschi1962iva,gautschi1963inverses,gautschi1974nei,Gautschi1978445}.
Gautschi obtained very precise bounds in \cite{gautschi1962iva,gautschi1963inverses},
but only for the case $\ell_{i}\leq2;\; i=1,\dots n$. In this note
we generalize these results for the arbitrary confluent configuration.
Our main result is as follows.
\begin{thm}
\label{thm:main-thm}Assume that the points $\left\{ x_{j}\right\} $
satisfy $\left|x_{j}\right|\leq1$ and also that they are $\delta$-separated,
i.e. $\left|x_{i}-x_{j}\right|\geq\delta>0$ for $i\neq j$. Denote
by $u_{j,k}$ the row with index $\ell_{1}+\dots+\ell_{j-1}+k+1$
of $V^{-1}$ (for $k=0,1,\dots,\ell_{j}-1$). Then the $\ell_{1}$-norm
of $u_{j,k}$ satisfies
\begin{equation}
\|u_{j,k}\|_{1}\leqslant\left(\frac{2}{\delta}\right)^{N}\frac{2}{k!}\left(\frac{1}{2}+\frac{N}{\delta}\right)^{\ell_{j}-1-k}.\label{eq:main-bound}
\end{equation}

\end{thm}
The proof of this theorem (see \prettyref{sec:proof-of-thm}) combines
original Gautschi's technique \cite{gautschi1963inverses} and the
well-known explicit expressions for the entries of $V^{-1}$ \cite{schappelle1972icv},
plus a technical lemma (\prettyref{sec:Technical-lemma}, \prettyref{lem:technical-lemma}).

In contrast with \cite{gautschi1963inverses}, the bound \eqref{eq:main-bound}
depends only on the separation distance and the combinatorial structure
of the problem. It shows that the norm grows polynomially with $\frac{2}{\delta}$
and exponentially with $N$.

\section{\label{sec:Technical-lemma}Technical lemma}
\begin{defn}
For $j=1,\dots,n$ let
\begin{equation}
h_{j}(x)=\prod_{i\neq j}(x-x_{i})^{-\ell_{i}}.\label{eq:h-def}
\end{equation}
\end{defn}
\begin{lem}
\label{lem:technical-lemma}The derivatives of $h_{j}$ at $x_{j}$
satisfy
\[
\left|h_{j}^{(t)}\left(x_{j}\right)\right|\leqslant N(N+1)\cdots(N+t-1)\delta^{-N-t}.
\]
\end{lem}
\begin{proof}
The proof has been kindly provided to us by did \cite{248559}. Assume
by induction that there exists a universal polynomial $P_{t}\left(N\right)$
of degree $t$ such that
\[
|h_{j}^{(t)}(x_{j})|\leqslant P_{t}(N)\,\delta^{-N-t}.
\]
For $t=0$ we have immediately $\left|h_{j}\left(x_{j}\right)\right|\leqslant\delta^{-N}$.
Now
\begin{equation}
h'_{j}(x)=h_{j}(x)\sum_{i\ne j}\frac{-\ell_{i}}{x-x_{i}}\label{eq:log-derivative}
\end{equation}
 Therefore we can apply the Leibnitz rule
\begin{eqnarray*}
h_{j}^{\left(t\right)}\left(x\right) & = & \left(\frac{h_{j}'}{h_{j}}h_{j}\right)^{\left(t-1\right)}=\sum_{k=0}^{t-1}{t-1 \choose k}h_{j}^{(k)}(x)\sum_{i\ne j}\frac{(-1)^{t-k-1}(t-k-1)!\ell_{i}}{(x-x_{i})^{t-k}},
\end{eqnarray*}
hence
\[
|h_{j}^{(t)}(x_{j})|\leqslant\sum_{k=0}^{t-1}{t-1 \choose k}|h_{j}^{(k)}(x_{j})|\sum_{i\ne j}\frac{(t-k-1)!\ell_{i}}{|x_{j}-x_{i}|^{t-k}}.
\]
This implies, together with the induciton hypothesis, that
\[
|h_{j}^{(t)}(x_{j})|\leqslant\sum_{k=0}^{t-1}{t-1 \choose k}\frac{P_{k}(N)}{\delta^{N+k}}\cdot\frac{(t-k-1)!N}{\delta^{t-k}}.
\]
So one can choose $P_{0}\left(N\right)=1$ and, for every $t\geq0,$
\[
P_{t}(N)=N\sum_{k=0}^{t-1}\frac{\left(t-1\right)!}{k!}P_{k}(N).
\]
This yields $P_{t}(N)=N(N+1)\cdots(N+t-1)$, which completes the proof.
\end{proof}

\section{\label{sec:proof-of-thm}Proof of \prettyref{thm:main-thm}}

By using a generalization of the Hermite interpolation formula (\cite{spitzbart1960generalization}),
it is shown in \cite{schappelle1972icv} that the components of the
row $u_{j,k}$ are just the coefficients of the polynomial
\[
\frac{1}{k!}\sum_{t=0}^{\ell_{j}-1-k}\frac{1}{t!}h_{j}^{(t)}(x_{j})(x-x_{j})^{k+t}\prod_{i\neq j}(x-x_{i})^{\ell_{i}}
\]
where $h_{j}\left(x\right)$ is given by \eqref{eq:h-def}.

By \cite[Lemma]{gautschi1962iva}, the sum of absolute values of the
coefficients of the polynomials $(x-x_{j})^{k+t}\prod_{i\neq j}(x-x_{i})^{\ell_{i}}$
is at most
\[
(1+|x_{j}|)^{k+t}\prod_{i\neq j}(1+|x_{i}|)^{\ell_{i}}\leqslant2^{N-(\ell_{j}-k-t)}.
\]

Therefore
\begin{eqnarray*}
\|u_{j,k}\|_{1} & \leqslant & \frac{1}{k!}\sum_{t=0}^{\ell_{j}-1-k}\frac{1}{t!}\frac{N(N+1)\cdots(N+t-1)}{{\delta^{N+t}}}2^{N-\ell_{j}+k+t}\\
 & = & \biggl(\frac{2}{\delta}\biggr)^{N}\frac{1}{{2^{\ell_{j}-k}k!}}\sum_{t=0}^{\ell_{j}-1-k}{\ell_{j}-1-k \choose t}\frac{N(N+1)\cdots(N+t-1)}{(\ell_{j}-k-t)\cdots(\ell_{j}-k-2)(\ell_{j}-k-1)}\biggl(\frac{2}{\delta}\biggr)^{t}\\
 & \leqslant & \biggl(\frac{2}{\delta}\biggr)^{N}\frac{1}{{2^{\ell_{j}-k}k!}}\biggl(1+\frac{2N}{\delta}\biggr)^{\ell_{j}-1-k}\\
 & = & \biggl(\frac{2}{\delta}\biggr)^{N}\frac{2}{k!}\biggl(\frac{1}{2}+\frac{N}{\delta}\biggr)^{\ell_{j}-1-k}
\end{eqnarray*}
which completes the proof.

\bibliographystyle{plain}
\bibliography{../../bibliography/all-bib}

\begin{thebibliography}{10}

\bibitem{batFullFourier}
D.~Batenkov.
\newblock {Complete Algebraic Reconstruction of Piecewise-Smooth Functions from
  Fourier Data}.
\newblock {\em preprint}, 2012.

\bibitem{batenkov2011accuracy}
D.~Batenkov and Y.~Yomdin.
\newblock {On the accuracy of solving confluent Prony systems}.
\newblock {\em To appear in SIAM J.Appl.Math.}
\newblock Arxiv preprint arXiv:1106.1137.

\bibitem{batyomAlgFourier}
D.~{Batenkov} and Y.~{Yomdin}.
\newblock {Algebraic Fourier reconstruction of piecewise smooth functions}.
\newblock {\em Mathematics of Computation}, 81:277--318, 2012.

\bibitem{bazan2000crv}
F.S.V. Baz{\'a}n.
\newblock {Conditioning of rectangular Vandermonde matrices with nodes in the
  unit disk}.
\newblock {\em SIAM Journal on Matrix Analysis and Applications}, 21:679, 2000.

\bibitem{beckermann2000condition}
B.~Beckermann.
\newblock {The condition number of real Vandermonde, Krylov and positive
  definite Hankel matrices}.
\newblock {\em Numerische Mathematik}, 85(4):553--577, 2000.

\bibitem{248559}
did (http://math.stackexchange.com/users/6179/did).
\newblock Evaluate a certain derivative.
\newblock Mathematics.
\newblock URL:http://math.stackexchange.com/q/248559 (version: 2012-12-01).

\bibitem{gautschi1962iva}
W.~Gautschi.
\newblock {On inverses of Vandermonde and confluent Vandermonde matrices}.
\newblock {\em Numerische Mathematik}, 4(1):117--123, 1962.

\bibitem{gautschi1963inverses}
W.~Gautschi.
\newblock {On inverses of Vandermonde and confluent Vandermonde matrices. II}.
\newblock {\em Numerische Mathematik}, 5(1):425--430, 1963.

\bibitem{gautschi1974nei}
W.~Gautschi.
\newblock {Norm estimates for inverses of Vandermonde matrices}.
\newblock {\em Numerische Mathematik}, 23(4):337--347, 1974.

\bibitem{Gautschi1978445}
W.~Gautschi.
\newblock {On inverses of Vandermonde and confluent Vandermonde matrices III}.
\newblock {\em Numerische Mathematik}, 29(4):445--450, 1978.

\bibitem{higham1996accuracy}
N.J. Higham.
\newblock {\em Accuracy and stability of numerical algorithms}.
\newblock Number~48. {SIAM}, 1996.

\bibitem{kalman1984gvm}
D.~Kalman.
\newblock {The generalized Vandermonde matrix}.
\newblock {\em Mathematics Magazine}, pages 15--21, 1984.

\bibitem{schappelle1972icv}
R.~Schappelle.
\newblock {The inverse of the confluent Vandermonde matrix}.
\newblock {\em IEEE Transactions on Automatic Control}, 17(5):724--725, 1972.

\bibitem{spitzbart1960generalization}
A.~Spitzbart.
\newblock {A generalization of Hermite's interpolation formula}.
\newblock {\em American Mathematical Monthly}, pages 42--46, 1960.

\end{thebibliography}

\end{document}